\documentclass[12pt,oneside]{amsart}

\usepackage{amsthm}
\usepackage{mathrsfs}
\usepackage{enumerate}
\usepackage{amsfonts}
\usepackage{verbatim}
\usepackage{amsbsy}
\usepackage{amsmath}
\usepackage{amssymb}
\usepackage{MnSymbol}
\usepackage{mathrsfs}
\usepackage{enumitem}
\usepackage{cite}
\usepackage{graphicx}

\newtheorem{theorem}{Theorem}[section]
\newtheorem{lemma}[theorem]{Lemma}
\newtheorem{observation}[theorem]{Observation}
\newtheorem{proposition}[theorem]{Proposition}
\newtheorem{conjecture}[theorem]{Conjecture}

\newtheorem{claim}[theorem]{Claim}
\newtheorem{question}[theorem]{Question}
\newtheorem{corollary}[theorem]{Corollary} 

\theoremstyle{definition}
\newtheorem{definition}[theorem]{Definition}

\theoremstyle{remark}
\newtheorem{remark}[theorem]{Remark}
\newtheorem{example}[theorem]{Example}
\newtheorem{case1}{Case}

\numberwithin{subcase}{case1}

\numberwithin{subcase2}{case2}

\makeatletter
\@addtoreset{case1}{theorem}
\makeatother

\makeatletter
\renewcommand*\env@matrix[1][*\c@MaxMatrixCols c]{%
  \hskip -\arraycolsep
  \let\@ifnextchar\new@ifnextchar
  \array{#1}}
\makeatother

\DeclareMathOperator{\core}{core}


\begin{document}
\title[\textit{M}\!-Ambiguity Sequences for Parikh Matrices]{\textit{M}\!-Ambiguity Sequences for Parikh Matrices and Their Periodicity Revisited}

\author{Ghajendran Poovanandran}
\address{School of Mathematical Sciences\\
Universiti Sains Malaysia\\
11800 USM\linebreak
Malaysia}
\address{School of Mathematics, Actuarial and Quantitative Studies\\
Asia Pacific University of Technology \& Innovation\\
Technology Park Malaysia, Bukit Jalil\\
57000 Kuala Lumpur, Malaysia}
\email{ghajendran@staffemail.apu.edu.my}

\author{Wen Chean Teh}
\address{School of Mathematical Sciences\\
Universiti Sains Malaysia\\
11800 USM\linebreak
Malaysia}
\email[Corresponding author]{dasmenteh@usm.my}

\begin{abstract}
The introduction of Parikh matrices by Mateescu et al.\! in 2001 has sparked numerous new investigations in the theory of formal languages by various researchers, among whom is \c{S}erb\v{a}nu\c{t}\v{a}. Recently, a decade-old conjecture by \c{S}erb\v{a}nu\c{t}\v{a} on the $M$\!-ambiguity of words was disproved, leading to new possibilities in the study of such words. In this paper, we investigate how selective repeated duplications of letters in a word affect the $M$\!-ambiguity of the resulting words. The corresponding $M$\!-ambiguity of those words are then presented in sequences, which we term as $M$\!-ambiguity sequences. We show that nearly all patterns of $M$\!-ambiguity sequences are attainable. Finally, by employing certain algebraic approach and some underlying theory in integer programming, we show that repeated periodic duplications of letters of the same type in a word results in an $M$\!-ambiguity sequence that is eventually periodic.
\end{abstract}

\maketitle

\section{Introduction}
The classical Parikh Theorem \cite{rP66}, which states that the Parikh vectors of all words from a context-free language form a semilinear set, established the Parikh mapping as a significant advancement in the theory of formal languages. The Parikh matrix mapping, introduced in \cite{MSSY01}, is a canonical generalization of the Parikh mapping. On top of dealing with the number of occurrences of individual letters (as in the case of Parikh vectors), the Parikh matrix of a word stores information on the number of occurrences of certain subwords in that word as well. The introduction of Parikh matrices has led to various new studies in the combinatorial study of words (for example, see \cite{aA07,AAP08,vS09,SS06,wT16,wT16b,aS10,SY10,AT16,GT16a,GT17,GT17b,MBS17,TSB18,GT17c,TK14}).

A word is $M$\!-ambiguous if and only if it shares the same Parikh matrix with another distinct word. In the pursuit of characterizing $M$\!-unambiguous words, \c{S}erb\v{a}nu\c{t}\v{a} proposed a conjecture in \cite{SS06} that the duplication of any letter in an $M$\!-ambiguous word will result in another $M$\!-ambiguous word. The conjecture was however overturned in \cite{GT17b} by a counterexample from the quaternary alphabet. 

In this work, we will show that by duplicating certain letters in a word, it is possible to continuously change the $M$\!-ambiguity of the resulting words. In fact, we will see that such changes in the $M$\!-ambiguity of a word can occur in nearly any pattern. Given an infinite sequence of words, obtained by repeatedly duplicating certain letters in the first word, we present the corresponding \mbox{$M$\!-ambiguity} of those words in what we term as an $M$\!-ambiguity sequence. This work also proposes an algebraic way to determine the $M$\!-ambiguity of a word. This algebraic approach is then used together with some underlying theory in integer linear programming to show that if we repeatedly duplicate---in a periodic manner---the letters of the same type in a word, the corresponding $M$\!-ambiguity sequence will be eventually periodic. 

The remainder of this paper is structured as follows. Section 2 provides the basic terminology and preliminaries. Section 3 highlights some previous results pertaining to the overturn of \c{S}erb\v{a}nu\c{t}\v{a}'s conjecture and serves the main motivation of this paper. After that, the central notion of our study, namely the $M$\!-ambiguity sequences, is introduced. It is then shown that nearly any pattern of $M$\!-ambiguity sequence can be realized. Section 4 mainly studies the periodicity of $M$\!-ambiguity sequences. In relative to that, an algebraic analysis to determine the $M$\!-ambiguity of a word is illustrated. Certain theories pertaining to rational polyhedra are then used together with the algebraic approach to prove a main result on the periodicity of $M$\!-ambiguity sequences. Our conclusion follows after that.

\section{Preliminaries}
We denote as follows---$\mathbb{R}$ is the set of real numbers, $\mathbb{Q}$ is the set of rational numbers, $\mathbb{Z}$ is the set of integers, $\mathbb{Z^+}$ is the set of positive integers and $\mathbb{Z}_{\ge 0}$ is the set of nonnegative integers.

Suppose $\Sigma$ is a finite nonempty alphabet. The set of all words over $\Sigma$ is denoted by $\Sigma^*$. The unique empty word is denoted by $\lambda$. Given two words $v,w\in\Sigma^*$, the concatenation of $v$ and $w$ is denoted by  $vw$. An \emph{ordered alphabet} is an alphabet $\Sigma=\{a_1, a_2, \dotsc,a_s\}$ with a total ordering on it. For example, if $a_1<a_2<\dotsb < a_s$, then we may write $\Sigma= \{a_1<a_2<\dotsb<a_s\}$. Conversely, if $\Sigma= \{a_1<a_2<\dotsb<a_s\}$, then $\{a_1, a_2, \dotsc,a_s\}$ is the \textit{underlying alphabet}. Frequently, we will abuse notation and use $\Sigma$ to stand for both the ordered alphabet and its underlying alphabet. Suppose $\Gamma\subseteq\Sigma$. The projective morphism $\pi_{\Gamma}:\Sigma^*\rightarrow\Gamma^*$ is defined by 
\begin{equation*}
\pi_{\Gamma}(a)=\begin{cases}
a, & \text{if } a\in\Gamma\\
\lambda, & \text{otherwise.}
\end{cases}
\end{equation*}

A word $v$ is a \emph{scattered subword} (or simply \emph{subword}) of $w\in \Sigma^*$ if and only if there exist $x_1,x_2,\dotsc, x_n$, $y_0, y_1, \dotsc,y_n\in \Sigma^*$ (possibly empty) such that $v=x_1x_2\dotsm x_n$ and  $w=y_0x_1y_1\dotsm y_{n-1}x_ny_n$. The number of occurrences of a word $v$ as a subword of $w$ is denoted by $\vert w\vert_v$. Two occurrences of $v$ are considered different if and only if they differ by at least one position of some letter. For example, $\vert bcbcc\vert_{bc}=5$ and $\vert aabcbc\vert_{abc}=6$.
By convention, $\vert w\vert_{\lambda}=1$ for all $w\in \Sigma^*$.

For any integer $k\geq 2$, let $\mathcal{M}_k$ denote the multiplicative monoid of $k \times k$ upper triangular matrices with nonnegative integral entries and unit diagonal.

\begin{definition} \cite{MSSY01}
Suppose $\Sigma=\{a_1<a_2<\cdots <a_k\}$ is an ordered alphabet. The \textit{Parikh matrix mapping} with respect to $\Sigma$, denoted by $\Psi_\Sigma$, is the morphism:
\begin{equation*}
\Psi_\Sigma:\Sigma^*\rightarrow\mathcal{M}_{k+1},
\end{equation*}
defined such that for every integer $1\le q\le k$, if $\Psi_\Sigma(a_q)=(m_{i,j})_{1\le i,j\le k+1}$, then
\begin{itemize}
\item $m_{i,i}=1$ for all $1\le i\le k+1$;
\item $m_{q,q+1}=1$; and 
\item all other entries of the matrix $\Psi_\Sigma(a_q)$ are zero.
\end{itemize} 
Matrices of the form $\Psi_\Sigma(w)$ for $w\in\Sigma^*$ are termed as  \textit{Parikh matrices}.
\end{definition}

\begin{theorem}\cite{MSSY01}\label{1206a}
Suppose $\Sigma=\{a_1<a_2< \dotsb<a_s\}$ is an ordered alphabet and $w\in \Sigma^*$. The matrix $\Psi_{\Sigma}(w)=(m_{i,j})_{1\leq i,j\leq s+1}$ has the following properties:
\begin{itemize}
\item $m_{i,i}=1$ for each $1\leq i \leq s+1$;
\item $m_{i,j}=0$ for each $1\leq j<i\leq s+1$;
\item $m_{i,j+1}=\vert w \vert_{a_ia_{i+1}\dotsm a_j    }$ for each $1\leq i\leq j \leq s$.
\end{itemize}
\end{theorem}

\begin{example}
Suppose $\Sigma=\{a<b<c<d\,\}$ and $w=abcdbc$.
Then \begin{align*}
\Psi_{\Sigma}(w)&=\Psi_{\Sigma}(a)\Psi_{\Sigma}(b)\Psi_{\Sigma}(c)\Psi_{\Sigma}(d)\Psi_{\Sigma}(b)\Psi_{\Sigma}(c)\\
&= \begin{pmatrix}
1 & 1 & 0 & 0 & 0\\
0 & 1 & 0 & 0 & 0\\
0 & 0 & 1 & 0 & 0\\
0 & 0 & 0 & 1 & 0\\
0 & 0 & 0 & 0 & 1
\end{pmatrix}
\begin{pmatrix}
1 & 0 & 0 & 0 & 0\\
0 & 1 & 1 & 0 & 0\\
0 & 0 & 1 & 0 & 0\\
0 & 0 & 0 & 1 & 0\\
0 & 0 & 0 & 0 & 1
\end{pmatrix}
\dotsm
\begin{pmatrix}
1 & 0 & 0 & 0 & 0\\
0 & 1 & 0 & 0 & 0\\
0 & 0 & 1 & 1 & 0\\
0 & 0 & 0 & 1 & 0\\
0 & 0 & 0 & 0 & 1
\end{pmatrix}\\
&= \begin{pmatrix}
1 & 1 & 2 & 3 & 1\\
0 & 1 & 2 & 3 & 1\\
0 & 0 & 1 & 2 & 1\\
0 & 0 & 0 & 1 & 1\\
0 & 0 & 0 & 0 & 1
\end{pmatrix}
=\begin{pmatrix}
1 & \vert w\vert_a & \vert w\vert_{ab} & \vert w\vert_{abc} & \vert w\vert_{abcd}\\
0 &1 & \vert w\vert_b & \vert w\vert_{bc} & \vert w\vert_{bcd}\\
0 & 0 & 1 & \vert w\vert_c & \vert w\vert_{cd}\\
0 & 0 & 0 & 1 & \vert w\vert_{d}\\
0 & 0 & 0 & 0 & 1
\end{pmatrix}.
\end{align*}
\end{example}

\begin{definition}
Suppose $\Sigma$ is an ordered alphabet.
Two words $w,w'\in \Sigma^*$ are \emph{$M$\!-equivalent}, denoted by $w\equiv_Mw'$, iff $\Psi_{\Sigma}(w)=\Psi_{\Sigma}(w')$.
A word $w\in \Sigma^*$ is \emph{$M$\!-ambiguous} iff it is $M$\!-equivalent to another distinct word. Otherwise, $w$ is \mbox{\emph{$M$\!-unambiguous}}. For any word $w\in\Sigma^*$, we denote by $C_w$ the set of all words
that are $M$\!-equivalent to $w$.
\end{definition}

The following is a simple equivalence relation which involves the most evident rewriting rules that preserve $M$\!-equivalence (see \cite{AAP08}).

\begin{definition}\label{1906a}
Suppose $\Sigma=\{a_1<a_2< \dotsb<a_s\}$ is an ordered alphabet. Two words $w,w'\in\Sigma^*$ are \mbox{\emph{$1$-equivalent}}, denoted by $w\equiv_1w'$, iff $w'$ can be obtained from $w$ by applying finitely many rewriting rules of the following form:
\begin{equation*}
xa_ka_ly\rightarrow xa_la_k y \text{ where } x,y\in\Sigma^* \text{ and } \vert k-l\vert \geq 2.
\end{equation*}
\end{definition}

\begin{definition}\cite{TK14}
Suppose $\Sigma$ is an alphabet and $v,w\in\Sigma^*$. The \textit{v-core} of $w$, denoted by $\core_v(w)$, is the unique subword $w'$ of $w$ such that $w'$ is the subword of shortest length which satisfies $|w'|_v=|w|_v$.
\end{definition}

\begin{proposition}\label{CoreabcMiddle}\cite{GT16a}
Suppose $\Sigma=\{a<b<c\}$ and $w\in\Sigma^*$ with $|w|_{abc}\ge 1$. Then, $w\equiv_1u\core_{abc}(w)v$ for some unique $u\in\{b,c\}^*$ and $v\in\{a,b\}^*$.
\end{proposition}

\section{Attainable Patterns of \textit{M}\!-ambiguity Sequences}

The following conjecture was proposed by \c{S}erb\v{a}nu\c{t}\v{a} in \cite{SS06} as an open \mbox{problem} pertaining to $M$\!-ambiguity of words.

\begin{conjecture}\label{SerbConj}
Suppose $\Sigma$ is an ordered alphabet. For any $u,v\in\Sigma^*$ and $a\in\Sigma$, if $uaav$ is $M$\!-unambiguous, then $uav$ is $M$\!-unambiguous as well. Equivalently, if $uav$ is $M$\!-ambiguous, then $uaav$ is also $M$\!-ambiguous.
\end{conjecture}

The above conjecture holds for the binary and ternary alphabets. (For exhaustive lists of $M$\!-unambiguous binary and ternary words, readers are referred to \cite[Theorem~3]{MS04} and \cite[Theorem~A.1]{SS06} respectively.)
On the contrary, for the quarternary alphabet, it was shown in \cite{GT17b} that the conjecture is invalid. The counterexample given was the $M$\!-ambiguous word $cbcbabcdcbabcbc$ (which is \mbox{$M$\!-equivalent} to the word $bccabcbdbcbaccb$). The following result was then proven, thus overturning the conjecture.

\begin{theorem}\cite{GT17b}\label{ConjectureOverturn}
The word $w=cbcbabc^ndcbabcbc$ is $M$\!-unambiguous with respect to $\Sigma=\{a<b<c<d\}$ for every integer $n>1$.
\end{theorem}

At this point, it is natural for one to ask Question~\ref{MainQues}, which is in a more general setting. 

\begin{definition}\cite{SS06}
Suppose $\Sigma$ is an alphabet and $w\in \Sigma^*$. Suppose $w=a^{p_1}_1a^{p_2}_2\cdots a^{p_n}_n$ such that $a_i\in\Sigma$ and $p_i>0$ for all $1\le i\le n$ with $a_i\neq a_{i+1}$ for all $1\le i\le n-1$. The \textit{print} of $w$, denoted by $pr(w)$, is the word $a_1a_2\cdots a_n$. 
\end{definition}

\begin{definition}
Suppose $\Sigma$ is an ordered alphabet and $w,w'\in\Sigma^*$. We write $w\dashv w'$ iff $w=uav$ and $w'=uaav$ for some $u,v\in\Sigma^*$ and $a\in\Sigma$.
\end{definition}

\begin{question}\label{MainQues}
Suppose $\Sigma=\{a<b<c<d\}$. Consider an infinite sequence of words $w_i\in\Sigma^*,i\ge 0$ such that $pr(w_0)=w_0$ and 
\begin{equation*}
w_0\dashv w_1\dashv w_2\dashv\dotsb.
\end{equation*}
In what patterns can the $M$\!-ambiguity of these words sequentially change?
\end{question}

In the spirit of answering the above question, we define the following notion.
\begin{definition}
Suppose $\Sigma$ is an ordered alphabet. Let $\varphi=\{w_i\}_{i\ge 0}$ be a sequence of words over $\Sigma$ such that for all integers $i\ge 0$, we have $w_i\dashv w_{i+1}$. We say that a sequence $\{m_i\}_{i\ge 0}$ is the \textit{$M$\!-ambiguity sequence corresponding to $\varphi$}, denoted by $\Theta_\varphi$, if and only if for every integer $i\ge 0$, we have $m_i\in\{A,U\}$ such that if $m_i=A$, then $w_i$ is $M$\!-ambiguous; otherwise if $m_i=U$, then $w_i$ is $M$\!-unambiguous.
\end{definition}

By the above definition, one can see that Question~\ref{MainQues} actually asks for the attainable patterns of $M$\!-ambiguity sequence, where the associated sequence of words starts with a print word. The following two examples, first presented in \cite{GT17e}, provide a partial answer to this question.

For the remaining part of this section, we fix $\Sigma=\{a<b<c<d\}$. Whenever the \mbox{$M$\!-ambiguity} of
a word is mentioned, it is understood that it is with \mbox{respect to $\Sigma$.}

\begin{example}\label{Example1}
For each integer $n\ge 1$, let $w_{n,m}=c^nbcbabc^mdcbabcbc$. If $m=n$, then $w_{n,m}$ is \mbox{$M$\!-ambiguous} as it is $M$\!-equivalent to the word $bc^{n+1}abc^nbdbcbaccb$. If $m=n+1$, then $w_{n,m}$ is $M$\!-unambiguous \cite[Theorem~3.6]{GT17e}.

Therefore, if one wants to obtain a sequence $\varphi$ of words (where the first word is a print word) such that $\Theta_\varphi=A,U,A,U,A,U,\cdots$, the duplication of letters can be carried out in the following manner:
$$w_{1,1},\, w_{1,2},\, w_{2,2},\, w_{2,3},\, w_{3,3},\, w_{3,4},\, \cdots .$$
\end{example}

\begin{example}\label{Example2}
The words $cbabcdcbabc$ and $cbabcdcbabbc$ are $M$\!-unambiguous (computationally verified). By duplicating the first letter $b$ in that word, we obtain the $M$-ambiguous word $cbbabcdcbabbc$ (it is \mbox{$M$\!-equivalent} to the word $bcabcbdbcbacb$). For each integer $n\ge 1$, let $w_{n,m}=c^nbbabc^mdcbabbc$. If $m=n,$ then $w_{n,m}$ is \mbox{$M$-ambiguous} as it is $M$\!-equivalent to the word $bc^nabc^nbdbcbacb$.  If $m=n+1$, then $w_{n,m}$ is $M$\!-unambiguous \cite[Theorem~3.7]{GT17e}.

Therefore, if one wants to obtain a sequence $\varphi$ of words (where the first word is a print word) such that $\Theta_\varphi=U,U,A,U,A,U,A,\cdots$, the duplication of letters can be carried out in the following manner:
$$cbabcdcbabc, cbabcdcbabbc, w_{1,1},\, w_{1,2},\, w_{2,2},\, w_{2,3},\, w_{3,3},\, w_{3,4},\, \cdots .$$
\end{example}

\begin{remark}\label{RemReason}
Example~\ref{Example1} and Example~\ref{Example2} shows that $M$\!-ambiguity sequences with alternating $A$ and $U$ are attainable. In contrast to Example~\ref{Example1}, the word $w_{1,1}$ in Example~\ref{Example2} is not a print word. That is why we needed the word $cbabcdcbabc$ to begin the sequence, followed by $cbabcdcbabbc$, before we reach $w_{1,1}$.
\end{remark}

We now generalize the words used in Example~\ref{Example1} and Example~\ref{Example2} to provide a more nearly complete answer---almost any pattern of \mbox{$M$\!-ambiguity} sequence is attainable. For that, we need the following observations and theorems as a basis.

\begin{observation}\label{BasisObs1}
For all positive integers $n$ and $p$, the word $c^nbcbabc^ndcbabcbc^p$ is \mbox{$M$\!-ambiguous} as it is \mbox{$M$\!-equivalent} to the word $bc^{n+1}abc^nbdbcbaccbc^{p-1}$. 
\end{observation}

The proof of the following result closely resembles that of Theorem~3.6 in \cite{GT17e}, yet we include it here for completeness.

\begin{theorem}\label{BasisTheo1}
The word $w=c^nbcbabc^mdcbabcbc^p$ is $M$\!-unambiguous for all integers $n\geq 1$, $p\ge 1$ and $m\ge n+1$.
\end{theorem}
\begin{proof}
We argue by contradiction. Fix integers $n\geq 1$, $p\ge 1$ and $m\ge n+1$. Assume that $w$ is \mbox{$M$\!-ambiguous.} Then, \mbox{$w\equiv_Mw'$} for some $w'\in\Sigma^*$ such that $w'\neq w$. It follows that $\pi_{\{a,b\}}(w)\equiv_M\pi_{\{a,b\}}(w')$, $\pi_{\{b,c\}}(w)\equiv_M\pi_{\{b,c\}}(w')$ and $\pi_{\{c,d\}}(w)\equiv_M\pi_{\{c,d\}}(w')$. Note that $\pi_{\{c,d\}}(w)=c^{n+m+1}dc^{p+2}$ is \mbox{$M$\!-unambiguous}, thus $\pi_{\{c,d\}}(w)=\pi_{\{c,d\}}(w')$. Meanwhile, $\pi_{\{a,b\}}(w)=bbabbabb$. Thus, $\pi_{\{a,b\}}(w')$ is either $bbabbabb$, $babbbbab$, $bbbaabbb$, or $abbbbbba$.

Write $w=\underbrace{c^nbcbabc^{m}}_{v_1}d\underbrace{cbabcbc^p}_{v_2}$ and $w'=v_1'dv_2'$, where $v_1',v_2'\in\{a,b,c\}^*$. (Note that $v_1$ and $v_2$ are both $M$\!-unambiguous\footnote{See Theorem A.1 in \cite{SS06}.} as this fact will be needed later in this proof.) Since $|w|_x=|w'|_x$ for every $x\in\{abcd,bcd,cd\}$, it follows that $|v_1|_y=|v_1'|_y$ for every $y\in\{abc,bc,c\}$. Furthermore, since $|v_1|_c+|v_2|_c=|w|_c=|w'|_c=|v_1'|_c+|v_2'|_c$, we have $|v_2|_c=|v_2'|_c$.

Note that $|w'|_{bc}=|v_1'|_{bc}+|v_1'|_b|v_2'|_c+|v_2'|_{bc}=|v_1|_{bc}+|v_1'|_b\cdot (p+2)+|v_2'|_{bc}$. At the same time,  $|w'|_{bc}=|w|_{bc}=|v_1|_{bc}+|v_1|_b|v_2|_c+|v_2|_{bc}=|v_1|_{bc}+3\cdot (p+2)+(3p+2)=|v_1|_{bc}+6p+8$. Thus,
\begin{equation}
|v_1'|_b\cdot (p+2)+|v_2'|_{bc}=6p+8.\tag{$\star$}
\end{equation}
Meanwhile, we have $|v_1'|_b\le |w'|_b=|w|_b=6$. If $|v_1'|_b=6$, then $|v_2'|_{bc}=-4$, which is impossible. Thus $|v_1'|_b\le 5$. Also, since $|v_1'|_{abc}=|v_1|_{abc}=m\ge n+1$, it follows that $|\core_{abc}(v_1')|_b\ge 1$. 

\begin{case1}$\pi_{\{a,b\}}(w')=bbabbabb$.\\
Since $|v_1'|_b\le 5$, $|\core_{abc}(v_1')|_b\ge 1$ and $\pi_{\{a,b\}}(w')=\pi_{\{a,b\}}(v_1')\pi_{\{a,b\}}(v_2')$, it follows that $\pi_{\{a,b\}}(v_1')\in\{bbab,bbabb,bbabba,bbabbab\}$.

Assume $\pi_{\{a,b\}}(v_1')=bbabbab$. Then $|v_1'|_{ab}=4$. Furthermore, as $|v_1'|_b=5$, it holds by $(\star)$ that $|v_2'|_{bc}=p-2$. Note that $\pi_{\{a,b\}}(v_2')=b$, therefore $|v_2'|_a=0$ and consequently $|v_2'|_{abc}=0$. Thus
\begin{align*}
|w'|_{abc}={}&|v_1'|_{abc}+|v_1'|_{ab}|v_2'|_c+|v_1'|_a|v_2'|_{bc}+|v_2'|_{abc}\\
={}&|v_1'|_{abc}+|v_1'|_{ab}|v_2'|_c+|v_1'|_a|v_2'|_{bc}\\
={}&m+4\cdot (p+2)+2\cdot (p-2)\\
={}&m+4p+8+2p-4\\
={}&m+6p+4.
\end{align*}
That is to say, $|w'|_{abc}=m+6p+4<m+6p+5=|w|_{abc}$, which is a contradiction.

Assume $\pi_{\{a,b\}}(v_1')\in\{bbabb,bbabba\}$. Then, $|v_1'|_{ab}=2$. Furthermore, as $|v_1'|_b=4$, it holds by $(\star)$ that $|v_2'|_{bc}=2p$. Note that if $\pi_{\{a,b\}}(v_1')=bbabb$, then $\pi_{\{a,b\}}(v_2')=abb$. Consequently $|v_2'|_a=1$ and therefore $|v_2'|_{abc}=|v_2'|_{bc}$. Otherwise if $\pi_{\{a,b\}}(v_1')=bbabba$, then $\pi_{\{a,b\}}(v_2')=bb$. Consequently, $|v_2'|_a=0$ and therefore $|v_2'|_{abc}=0$ as well. In both cases, we have $|v_2'|_{abc}=|v_2'|_a|v_2'|_{bc}$. Thus
\begin{align*}
|w'|_{abc}={}&|v_1'|_{abc}+|v_1'|_{ab}|v_2'|_c+|v_1'|_a|v_2'|_{bc}+|v_2'|_{abc}\\
={}&|v_1'|_{abc}+|v_1'|_{ab}|v_2'|_c+|v_1'|_a|v_2'|_{bc}+|v_2'|_a|v_2'|_{bc}\\
={}&|v_1|_{abc}+|v_1'|_{ab}|v_2'|_c+(|v_1'|_a+|v_2'|_a)\cdot |v_2'|_{bc}\\
={}&m+2\cdot (p+2)+2\cdot 2p\\
={}&m+2p+4+4p\\
={}&m+6p+4.
\end{align*}
Similar to the case $\pi_{\{a,b\}}(v_1')=bbabbab$, we have $|w'|_{abc}=m+6p+4<m+6p+5=|w|_{abc}$, which is a contradiction.

Thus $\pi_{\{a,b\}}(v_1')=bbab$. We have $\pi_{\{a,b\}}(v_1')=\pi_{\{a,b\}}(v_1)$, therefore $|v_1'|_y=|v_1|_y$ for every $y\in\{a,b,ab\}$. As we already know that $|v_1'|_y=|v_1|_y$ for every $y\in\{abc,bc,c\}$, it follows that 
$v_1'\equiv_M v_1$ with respect to $\{a<b<c\}$. However, $v_1$ is \mbox{$M$\!-unambiguous}, thus $v_1'=v_1$. Consequently, $v_2'\equiv_M v_2$ with respect to $\{a<b<c\}$ by the left invariance of $M$\!-equivalence. Similarly, $v_2$ is \mbox{$M$\!-unambiguous}, thus $v_2'=v_2$. Therefore $w'=w$, which is a contradiction.
\end{case1}

\begin{case1}$\pi_{ab}(w')=babbbbab$.\\
By similar reasoning as in Case~1, we have $\pi_{\{a,b\}}(v_1')=\{bab, babb, babbb, babbbb\}$. In all four cases, $|v_1'|_a=1$. Also, note that $|v_1'|_c=|v_1|_c=n+m+1$, $|v_1'|_{bc}=|v_1|_{bc}=3m+1$ and $|v_1'|_{abc}=|v_1|_{abc}=m$.

By Proposition~\ref{CoreabcMiddle}, it holds that $v_1'\equiv_1u_1\core_{abc}(v_1')u_2$ for some unique $u_1\in\{b,c\}^*$ and $u_2\in\{a,b\}^*$. Since $\pi_{ab}(v_1')\in\{bab,babb,babbb,babbbb\}$ and $a$ is a prefix of $\core_{abc}(v_1')$, it follows that $\vert u_1\vert_b=1$.
Also, note that $|v_1'|_{bc}=|u_1|_{bc}+|u_1|_b|\core_{abc}(v_1')|_c+|\core_{abc}(v_1')|_{bc}$. Since $|\core_{abc}(v_1')|_{abc}=\vert v_1'\vert_{abc}=m$,
$a$ is a prefix of $\core_{abc}(v_1')$, and that is the only $a$ in $\core_{abc}(v_1')$, it follows that $|\core_{abc}(v_1')|_{bc}=m$.
Additionally, since $|u_1|_{bc}+ |u_1|_{cb}=|u_1|_b|u_1|_c$, it follows that $|u_1|_{bc} \le |u_1|_b|u_1|_c= |u_1|_c$. Therefore,
$|v_1'|_{bc}
\le |u_1|_{c}+|\core_{abc}(v_1')|_c+m
=|v_1'|_c+m
=n+m+1+m
=n+2m+1$. 
Consequently, $3m+1=|v_1|_{bc}=|v_1'|_{bc}\le n+2m+1$, which reduces to $m\le n$. Thus a contradiction occurs.
\end{case1}

\begin{case1}$\pi_{ab}(w')=bbbaabbb$.\\
This case is impossible. Observe that $|v_1'|_b=3+|\core_{abc}(v_1')|_b$. Since $|v_1'|_b\le 5$ and $|\core_{abc}(v_1')|_b\ge 1$, it follows that $\pi_{ab}(v_1')\in\{bbbaab,bbbaabb\}$. 

If $\pi_{ab}(v_1')=bbbaabb$, then $|v_1'|_b=5$ and consequently $|v_2'|_{bc}=p-2$ due to $(\star)$. Correspondingly, we have $|w'|_{abc}
=|v_1'|_{abc}+|v_1'|_{ab}|v_2'|_c+|v_1'|_a|v_2'|_{bc}+|v_2'|_{abc}
=m+4\cdot (p+2)+2\cdot (p-2)+0=m+6p+4$.
On the other hand, if $\pi_{ab}(v_1')=bbbaab$, then $|v_1'|_b=4$ and consequently $|v_2'|_{bc}=2p$ due to $(\star)$. Correspondingly, we have $|w'|_{abc}
=|v_1'|_{abc}+|v_1'|_{ab}|v_2'|_c+|v_1'|_a|v_2'|_{bc}+|v_2'|_{abc}
=m+2\cdot (p+2)+2\cdot (2p)+0=m+6p+4$ as well. 
In both cases, $|w'|_{abc}<m+6p+5=|w|_{abc}$, which is a contradiction.
\end{case1}

\begin{case1}$\pi_{ab}(w')=abbbbbba$.\\
This case is trivially impossible. Note that $|v_1'|_{bc}=|v_1|_{bc}=3m+1$. Consequently, $|v_1'|_{abc}=1\cdot|v_1'|_{bc}=3m+1$. However, $|v_1'|_{abc}=|v_1|_{abc}=m$, thus a contradiction.
\qedhere
\end{case1}
\end{proof}

Observation~\ref{BasisObs1} and Theorem~\ref{BasisTheo1} allow us to generate sequences of words (starting with a print word) such that the first word is \mbox{$M$\!-ambiguous} and the \mbox{$M$\!-ambiguity} of the remaining words sequentially change in an arbitrary pattern. This is illustrated by the following example.

\begin{example}\label{ExFibo}
Consider the Fibonacci sequence $0,1,1,2,3,5,8,13,\cdots$.
Suppose we want to generate a sequence of words realizing the $M$\!-ambiguity sequence where the number of terms $U$ between two consecutive terms $A$ follows the Fibonacci sequence---i.e, $A,A,U,A,U,A,U,U,A,\cdots$. 

For all integers $n,m,p\ge 1$, let $w_{n,m,p}=c^nbcbabc^mdcbabcbc^p$. By Observation~\ref{BasisObs1}, if $m=n$, then $w_{n,m,p}$ is $M$\!-ambiguous for any $p\ge 1$. By Theorem~\ref{BasisTheo1}, if $m=n+1$, then $w_{n,m,p}$ is $M$\!-unambiguous for any $p\ge 1$. Thus, it remains to duplicate the letters in the following manner:
\begin{equation*}
w_{1,1,1},w_{1,1,2},w_{1,2,1},w_{2,2,1},w_{2,3,1},w_{3,3,1},w_{3,4,1},w_{3,4,2},w_{4,4,2},\cdots
\end{equation*}
(Notice that whenever we need to retain the preceding term, we increase the power $p$ by one---that is to duplicate the last letter $c$.)
\end{example}

On the other hand, to generate similar sequences of words such that the first word is \mbox{$M$\!-unambiguous}, we need the following observation and result.

\begin{observation}\label{BasisObs2}
For all positive integers $n$ and $p$, the word $c^nbbabc^ndcbabbc^p$ is \mbox{$M$\!-ambiguous} as it is \mbox{$M$-equivalent} to the word $bc^nabc^nbdbcbacbc^{p-1}$.
\end{observation}
 
\begin{theorem}\label{BasisTheo2}
The word $c^nbbabc^mdcbabbc^p$ is \mbox{$M$\!-unambiguous} for all integers $n\geq 1$, $p\ge 1$ and $m\ge n+1$.
\end{theorem}
\begin{proof}
Argue similarly as in the proof of Theorem~\ref{BasisTheo1}. \qedhere
\end{proof}

\begin{remark}\label{RemDrawback}
When $n\!=\!m\!=\!p\!=\!1$, in contrast to the word in \mbox{Theorem~\ref{BasisTheo1},} the word in Theorem~\ref{BasisTheo2} is not a print word. Thus, similarly as in Example~\ref{Example2}, we need the $M$\!-unambiguous words $cbabcdcbabc$ and $cbabcdcbabbc$ on top of Observation~\ref{BasisObs2} and Theorem~\ref{BasisTheo2} to realize $M$\!-ambiguity sequences starting with $U$. However, this forces the first three terms to be $U$, $U$, and $A$ before we can change the terms arbitrarily.
\end{remark}

\section{Periodicity of \textit{M}\!-ambiguity Sequences}

Consider the word $cbcbab\underline{c}dcbabcbc$ over the ordered alphabet $\{a<b<c<d\}$. By Theorem~\ref{ConjectureOverturn}, it holds that every duplication of the underlined letter $c$ in that word gives rise to an $M$\!-unambiguous word. Thus for the sequence of words $\varphi=\{w_i\}_{i\ge 1}$ such that $w_i=cbcbabc^idcbabcbc$, we have $\Theta_{\varphi}=A,U,U,U,\cdots$. 

We see that the sequence $\Theta_{\varphi}$ is eventually periodic with its period being one. Thus we seek to know whether the periodicity of an $M$\!-ambiguity sequence is a trait in the case of duplicating a single letter in a word. We formulate this question formally as follows.

\begin{question}\label{QuesMotivation}
Suppose $\Sigma$ is an ordered alphabet. Let $\varphi=\{w_i\}_{i\ge 1}$ be a sequence of words over $\Sigma$ such that for every integer $k\ge 1$, we have $w_k=xa^ky$ for some $x,y\in\Sigma^*$ and $a\in\Sigma$. Is the sequence $\Theta_{\varphi}$ eventually periodic?
\end{question}

In the spirit of answering the above question, we first present a way to determine the $M$\!-ambiguity of a word---by transforming it to a problem of solving \mbox{systems} of linear equalities. To illustrate this, we analyze the word considered in Theorem~\ref{ConjectureOverturn} and deduce that it is $M$\!-unambiguous for every integer $n>1$. 

Let $\Sigma=\{a<b<c<d\}$ and consider the word $w=cbcbabc^ndcbabcbc$, where $n$ is a nonnegative integer. If a word $w'\in\Sigma^*$ is $M$\!-equivalent to $w$, then $\pi_{\{a,b,d\}}(w')\equiv_M\pi_{\{a,b,d\}}(w)$ with respect to \mbox{$\{a<b<d\}$}. Since $\pi_{\{a,b,d\}}(w)=bbabdbabb$, it follows that for such a word $w'$, the projection $\pi_{\{a,b,d\}}(w')$ must be one of the following:
\begin{equation*}\tag{$\ast\ast$}
\begin{aligned}
                      &dbbabbabb,bdbabbabb,bbdabbabb,bbadbbabb,bbabdbabb,bbabbdabb,\\
                      &bbabbadbb,bbabbabdb,bbabbabbd,dbabbbbab,bdabbbbab,badbbbbab,\\
                      &babdbbbab,babbdbbab,babbbdbab,babbbbdab,babbbbadb,babbbbabd,\\
                      &dbbbaabbb,bdbbaabbb,bbdbaabbb,bbbdaabbb,bbbadabbb,bbbaadbbb,\\
                      &bbbaabdbb,bbbaabbdb,bbbaabbbd,dabbbbbba,adbbbbbba,abdbbbbba\\
                      &abbdbbbba,abbbdbbba,abbbbdbba,abbbbbdba,abbbbbbda,abbbbbbad.
\end{aligned}
\end{equation*}

Consider the scenario $\pi_{\{a,b,d\}}(w')=\pi_{\{a,b,d\}}(w)=bbabdbabb$. Then $$w'=c^{x_1}\boldsymbol{b}c^{x_2}\boldsymbol{b}c^{x_3}\boldsymbol{a}c^{x_4}\boldsymbol{b}c^{x_5}\boldsymbol{d}c^{x_6}\boldsymbol{b}c^{x_7}\boldsymbol{a}c^{x_8}\boldsymbol{b}c^{x_9}\boldsymbol{b}c^{x_{10}}$$
for some nonnegative integers $x_i\,(1\le i\le 10)$. Since $w'\equiv_Mw$, it follows that 
\vspace{1em}\begin{center}
\setlength\tabcolsep{2pt}
\begin{tabular}{r c l c l c l}
$x_1+x_2+x_3+x_4+x_5+x_6+x_7+x_8+x_9+x_{10}$ & $=$ & $|w'|_c$ & $=$ & $|w|_c$ & $=$ & $n+5$,\\
$x_2+2x_3+2x_4+3x_5+3x_6+4x_7+4x_8+5x_9+6x_{10}$ & $=$ & $|w'|_{bc}$ & $=$ & $|w|_{bc}$ & $=$ & $3n+15$,\\
$x_1+x_2+x_3+x_4+x_5$ & $=$ & $|w'|_{cd}$ & $=$ & $|w|_{cd}$ & $=$ & $n+2$,\\
$x_5+x_6+2x_7+2x_8+4x_9+6x_{10}$ & $=$ & $|w'|_{abc}$ & $=$ & $|w|_{abc}$ & $=$ & $n+11$,\\
$x_2+2x_3+2x_4+3x_5$ & $=$ & $|w'|_{bcd}$ & $=$ & $|w|_{bcd}$ & $=$ & $3n+1$,\\
$x_5$ & $=$ & $|w'|_{abcd}$ & $=$ & $|w|_{abcd}$ & $=$ & $n$.
\end{tabular}
\end{center} 

Solving the above system of linear equalities, we obtain the solution set
\begin{equation*}
\begin{split}
x_1 & = 1+x_3+x_4, \\
x_2 & = 1-2x_3-2x_4,\\
x_5 & = n,\\
x_6 & = 1,\\
x_7 & = -1-x_8+x_{10},\\
x_9 & = 3-2x_{10}.
\end{split}
\end{equation*} 

By imposing the constraints $x_i\ge 0\,(1\le i\le 10)$, we now have the system of linear inequalities
\begin{equation*}\tag{$\ast\ast\ast$}
\begin{split}
x_3+x_4&\ge -1, \\
2x_3+2x_4&\le 1,\\
n&\ge 0,\\
x_8-x_{10}&\le -1,\\
2x_{10}&\le 3\\
x_3,x_4,x_8,x_{10}&\ge 0.
\end{split}
\end{equation*}

From the above system of linear inequalities, notice that the only possible value of $x_{10}$ is 1 and therefore $x_8=0$. Also, observe that it can only be the case that $x_3=x_4=0$. By the system of linear equations before that, it follows that $x_1=x_2=x_6=x_9=1$, $x_5=n$ and $x_7=0$. As a result, we have $w'=cbcbabc^ndcbabcbc$. However, notice that $w'=w$, thus this scenario does not imply that $w$ is \mbox{$M$\!-ambiguous} for every nonnegative \mbox{integer $n$.}

Next, consider the scenario $\pi_{\{a,b,d\}}(w')=babbbdbbab$. Analyzing similarly as above, we obtain the solution set
\begin{equation*}
\begin{split}
x_1 & = 1-n, \\
x_2 & = 1-x_3+x_5+n,\\
x_4 & = -2x_5+n,\\
x_6 & = -2+x_8+x_9,\\
x_7 & = 5-2x_8-2x_9,\\
x_{10} & = 0.
\end{split}
\end{equation*} 
and the system of linear inequalities
\begin{equation*}
\begin{split}
n&\le 1, \\
x_3-x_5&\le 1,\\
2x_5-n&\le 0,\\
x_8+x_9&\ge 2,\\
2x_8+2x_9&\le 5\\
x_3,x_5,x_8,x_9&\ge 0.
\end{split}
\end{equation*}

By some simple analysis, one can see that for $n=0$ or $n=1$, integral solutions exist for the above system---each of them gives rise to a word $w'$ that is distinct from $w$. This implies that when $n=0$ or $n=1$, the word $w$ is \mbox{$M$\!-ambiguous}. However, when $n>1$, there are no integral solutions, with such $n$, satisfying the system.

Arguing like this, one can see that each possibility of $\pi_{\{a,b,d\}}(w')$ in $(\ast\ast)$ leads to a system of linear equations and inequalities. Every such system can then be analyzed similarly as in above (thus we omit the details of the remaining computations). In our case here, when $n>1$, all the remaining $34$ systems lead to no solutions. Thus, we conclude that the word $w=cbcbabc^ndcbabcbc$ is $M$\!-unambiguous for all integers $n>1$.

\begin{remark}\label{RemForProof}
Suppose $\Sigma$ is an ordered alphabet. For a general word $$x_1a^{k}x_2a^{k}\cdots x_{j-1}a^{k}x_j$$ where $x_1,x_2,\cdots , x_j\in\Sigma^*$, $a\in\Sigma$, and $k$ is a positive integer, the above algebraic \mbox{analysis} can be used to determine the values of $k$ such that the word is \mbox{$M$\!-ambiguous}. The corresponding (finitely many) systems of linear equalities and \mbox{inequalities} are rational.
We will need this observation for the proof of Theorem~\ref{TheoEventPeriod1} later.
\end{remark}

Next, we need the following notion and known result, which in turn will be used to prove a lemma necessary for our purpose.

\begin{definition}\label{DefPoly}
Suppose $n$ is a positive integer. A set $P\subseteq\mathbb{R}^n$ is a \textit{rational polyhedron} if and only if $P=\{\boldsymbol{x}\in\mathbb{R}^n\,|\,A\boldsymbol{x}\ge \boldsymbol{b}\}$ for some matrix $A\in\mathbb{Q}^{m\times n}$ and vector $\boldsymbol{b}\in\mathbb{Q}^m$, where $m$ is a positive integer.
\end{definition}

The following result was deduced as Equation~19 in Chapter 16 of \cite{aS98}. We do not state the underlying details that lead to this result here as they are not essential for our purpose.

\begin{theorem}\cite{aS98}\label{DecompTheo}
Suppose $n$ is a positive integer. For any rational polyhedra $P\subseteq\mathbb{R}^n$, there exist vectors $\boldsymbol{x_1},\boldsymbol{x_2},\cdots ,\boldsymbol{x_r},\boldsymbol{y_1},\boldsymbol{y_2},\cdots ,\boldsymbol{y_s}\in\mathbb{Z}^n$ such that
\begin{equation*}
\begin{aligned}
\{\boldsymbol{x}\in P\cap\mathbb{Z}^n\}=&\{\lambda_1\boldsymbol{x_1}+\cdots +\lambda_r\boldsymbol{x_r}+\mu_1\boldsymbol{y_1}+\cdots +\mu_s\boldsymbol{y_s}\,|\,\lambda_1,\cdots ,\lambda_r,\mu_1,\cdots ,\mu_s \text{ are }\\
&\text{ nonnegative integers with }\lambda_1+\cdots+\lambda_r=1\}.
\end{aligned}
\end{equation*}
\end{theorem}

We are now ready to prove our main lemma.
\begin{lemma}\label{LemMain}
Suppose $n$ is a positive integer. Let $P=\{\boldsymbol{x}\in\mathbb{R}^n\,|\,A\boldsymbol{x}\ge \boldsymbol{b}\}$ for some matrix $A\in\mathbb{Q}^{m\times n}$ and vector $\boldsymbol{b}\in\mathbb{Q}^{m}$ where $m$ is a positive integer. Choose an arbitary integer $1\le k\le n$. Let 
$$P_k=\{p\in\mathbb{Z}^+\,|\, p \text{ is the $k^{th}$\! component of some } \boldsymbol{x}\in P\cap\mathbb{Z}_{\ge 0}^n\}.$$ Suppose the set $P_k$ is infinite. Then, for some positive integer $d$ and nonempty set $T\subseteq[0,d)\cap\mathbb{Z}$, there exists a positive integer $N$ such that 
$$\{p\in P_k\,|\,p\ge N\}=\{p\in\mathbb{Z}^+\,|\,p\ge N \text{ and } p=dq+t \text{ for some } t\in T \text{ and integers }q\}.$$
\end{lemma}

\begin{proof}
Clearly, by Definition~\ref{DefPoly}, $P$ is a rational polyhedron. Therefore, by Theorem~\ref{DecompTheo}, there exist vectors $\boldsymbol{x_1},\boldsymbol{x_2},\cdots ,\boldsymbol{x_r},\boldsymbol{y_1},\boldsymbol{y_2},\cdots ,\boldsymbol{y_s}\in\mathbb{Z}^n$ such that
\begin{equation*}
\begin{aligned}
\{\boldsymbol{x}\in P\cap\mathbb{Z}^n\}=\{&\lambda_1\boldsymbol{x_1}+\cdots +\lambda_r\boldsymbol{x_r}+\mu_1\boldsymbol{y_1}+\cdots +\mu_s\boldsymbol{y_s}\,|\,\lambda_1,\cdots ,\lambda_r,\mu_1,\cdots ,\mu_s \text{ are }\\
&\text{nonnegative integers with }\lambda_1+\cdots+\lambda_r=1\}.
\end{aligned}
\end{equation*}

Fix an arbitary integer $1\le k\le n$. Suppose the set $P_k$ is infinite. Let $\boldsymbol{x}[i]$ denote the $i^{th}$ component of a vector $\boldsymbol{x}$. Then, for an arbitrary $p\in P_k$, it holds that
\begin{equation*}\tag{4.5.1}
\begin{aligned}
&p=\lambda_{1}\boldsymbol{x_1}[k]+\cdots +\lambda_{r}\boldsymbol{x_r}[k]+\mu_{1}\boldsymbol{y_1}[k]+\cdots +\mu_{s}\boldsymbol{y_s}[k] \text{ for some}\\
&\text{nonnegative integers }\lambda_{1},\cdots ,\lambda_{r},\mu_{1},\cdots ,\mu_{s}\text{ with }\lambda_{1}+\cdots+\lambda_{r}=1.
\end{aligned}
\end{equation*}

Assume $\boldsymbol{y_{i}}[k]$ is nonpositive for all integers $1\le i\le s$. Since $p$ is a positive integer and $\lambda_{1},\cdots ,\lambda_{r}$ are nonnegative integers with $\lambda_{1}+\cdots+\lambda_{r}=1$, it holds that $0<p\le \max\{\boldsymbol{x_i}[k]\,|\,1\le i\le r\}$. However, such values of integers $p$ are only finitely many, which is a contradiction as the set $P_k$ is infinite. Thus $\boldsymbol{y_{i}}[k]$ is positive for some integers $1\le i\le s$.

Choose an integer $I$ such that $\boldsymbol{y_{I}}[k]$ is positive. Let 
\begin{equation*}
T=\{t\in[0,\boldsymbol{y_{I}}[k])\cap\mathbb{Z}\,|\,t=p-\boldsymbol{y_{I}}[k]\cdot q \text{ for some }p\in P_k \text{ and integer } q\}.
\end{equation*}
For every $t\in T$, let
\begin{equation*}
p_t^*=\min\{p\in P_k\,|\,p=\boldsymbol{y_{I}}[k]\cdot q+t \text{ for some integer }q\}.
\end{equation*}
Then, by (4.5.1), it follows that
\begin{gather*}\tag{4.5.2}
\text{for every $t\in T$ and integer $j\ge 0$, we have  $p_t^*+j\cdot\boldsymbol{y_{I}}[k]\in P_k$.}
\end{gather*}

Let $N=\max\{p_t^*\,|\,t\in T\}$ and $d=\boldsymbol{y_{I}}[k]$. Then the forward inclusion clearly holds by the definition of $N,d$ and $T$. To show that the backward inclusion holds, fix an arbitrary $p\in\mathbb{Z}^+$ with $p\ge N$ such that $p=\boldsymbol{y_{I}}[k]\cdot q+t$ for some $t\in T$ and integer $q$.
Let $q^*_t$ be the integer such that $p^*_{t}=\boldsymbol{y_{I}}[k]\cdot q^*_{t}+t$. Note that $p^*_{t}\le N\le p$, thus $q^*_{t}\le q$. Notice that
\begin{equation*}
\begin{aligned}
 p&=\boldsymbol{y_{I}}[k]\cdot q+t\\
  &=\boldsymbol{y_{I}}[k]\cdot (q+q_t^*-q_t^*)+t\\
  &=\boldsymbol{y_{I}}[k]\cdot (q^*_{t}+(q-q_t^*))+t\\
  &=\boldsymbol{y_{I}}[k]\cdot q^*_{t} + t + \boldsymbol{y_{I}}[k]\cdot (q-q_t^*)\\
  &=p_t^*+\boldsymbol{y_{I}}[k]\cdot (q-q_t^*).
\end{aligned}
\end{equation*}
It remains to see that since $q-q^*_{t}\ge 0$, by (4.5.2), it holds that $p\in P_k$. Thus our conclusion holds.
\end{proof}

\begin{theorem}\label{TheoEventPeriod1}
Suppose $\Sigma$ is an ordered alphabet. Let $\varphi=\{w_k\}_{k\ge 1}$ be a sequence of words over $\Sigma$ such that for every integer $k\ge 1$, we have $$w_k=x_1a^kx_2a^k\cdots x_{j-1}a^kx_j$$ for some $x_1,x_2,\cdots x_j\in\Sigma^*$ and $a\in\Sigma$. Then, $\Theta_\varphi$ is eventually periodic.
\end{theorem}
\begin{proof}
In Remark~\ref{RemForProof}, we observe that for a word $w_k=x_1a^kx_2a^k\cdots x_{j-1}a^kx_j$ (as in the hypothesis) where $k$ is a positive integer, the algebraic analysis presented in the beginning of this section can be used to determine the values of $k$ such that $w_k$ is $M$\!-ambiguous. For the completeness of this proof, we will reiterate certain parts of the aforementioned analysis.

Let $\Gamma=\Sigma\backslash\{a\}$. Write $w_k$ in the form $a^{\gamma_1}\beta_1a^{\gamma_2}\beta_2\cdots a^{\gamma_{n}}\beta_{n}a^{\gamma_{n+1}}$ for some positive integer $n$, integers $\gamma_i\ge 0\,(1\le i\le n+1)$ and $\beta_i\in\Gamma \,(1\le i\le n)$. Note that $\beta_1\beta_2\cdots\beta_n=\pi_{\Gamma}(w_k)$.
Suppose there exists $w'\in\Sigma^*$ such that $w'\equiv_Mw_k$. Then 
$\pi_{\Gamma}(w')\equiv_M\pi_{\Gamma}(w_k)$. Each possibility of the projection $\pi_{\Gamma}(w')$ gives rise to a rational system of linear inequalities as in $(\ast\ast\ast)$, with $k$ being a variable in it (due to the constraint $k\ge 1$). Each such system, when solved for nonnegative integral solutions, contains the values of $k$ such that $w_k$ is $M$\!-equivalent to $w'$ with that projection. 
 
Assume $\pi_{\Gamma}(w')=\pi_{\Gamma}(w_k)=\beta_1\beta_2\cdots\beta_n$. Then $w'=a^{y_1}\beta_1a^{y_2}\beta_2\cdots a^{y_{n}}\beta_{n}a^{y_{n+1}}$ for some integers $y_i\ge 0\,(1\le i\le n+1)$. If $y_i=\gamma_i$ for every integer $1\le i\le n+1$, then $w'=w$. To avoid this, we impose the condition $y_i<\gamma_i$ or $y_i>\gamma_i$ for some integer $1\le i\le n+1$. Thus for every integer $1\le i\le n+1$, we consider two distinct systems of linear inequalities, each of them consisting of the ones obtained as in $(\ast\ast\ast)$, together with one of the conditions $y_i<\gamma_i$ or $y_i>\gamma_i$
---this gives a total of $2(n+1)$ systems of linear inequalities.
On the other hand, if $\pi_{\Gamma}(w')\neq\pi_{\Gamma}(w_k)$, then it is impossible for $w'$ to be the same word as $w$. Thus, for each such possibility of $\pi_{\Gamma}(w')$, it suffices to consider the system of linear inequalities obtained as in $(\ast\ast\ast)$
---this gives a total of $|C_{\pi_{\Gamma}(w_k)}-1|$ systems.

Let $N=2(n+1)+|C_{\pi_{\Gamma}(w_k)}-1|$. Let integers $1\le i\le N$ enumerate the systems of linear inequalities that we have and write each of them in the form $A_i\boldsymbol{y}\ge \boldsymbol{b_i}$ for some matrix $A_i\in\mathbb{Q}^{r\times s}$ and vector $\boldsymbol{b_i}\in\mathbb{Q}^{r}$, where $r$ is a positive integer and $s=n+2$.
For every integer $1\le i\le N$, let $P_i=\{\boldsymbol{y}\in\mathbb{R}^q\,|\,A_i\boldsymbol{y}\ge \boldsymbol{b_i}\}$ and let $\tau_i$ be the index such that the ${\tau_i}^{th}$ component of $\boldsymbol{y}$ corresponds to the variable $k$. Also, for every integer $1\le i\le N$, define the set
$$P_i^*=\{p\in\mathbb{Z}^+\,|\, p \text{ is the ${\tau_i}^{th}$\! component of some } \boldsymbol{y}\in P_i\cap\mathbb{Z}_{\ge 0}^q\}.$$
Notice that 
\begin{equation}\tag{4.6.1}
\text{the word } w_k \text{ is M\!-ambiguous if and only if } k\in\underset{1\le i\le N}{\bigcup} P_i^*.
\end{equation}

\begin{case1}The set $P_i^*$ is finite for every integer $1\le i\le N$.\\
Then the set $\underset{1\le i\le N}{\bigcup} P_i^*$ is finite as well. By (4.6.1), the word $w_k$ is $M$\!-ambiguous for only finitely many values of $k$.\! For every integer $k>\!\max\{k\,|\,w_k \text{ is $M$\!-ambiguous}\}$, the word $w_k$ is $M$\!-unambiguous. Therefore, $\Theta_\varphi$ is eventually periodic (with its period being one).
\end{case1}

\begin{case1}The set $P_i^*$ is infinite for some integer $1\le i\le N$.\\
Let $I=\{1\le i\le N\,|\,\text{the set }P_i^* \text{ is infinite}\}$. For every integer $i\in I$, by Lemma~\ref{LemMain}, it follows that for some positive integer $d_i$ and nonempty set $T_i\subseteq[0,d_i)\cap\mathbb{Z}$, there exists a positive integer $M_i$ such that 
$$\{p\in P_i^*\,|\,p\ge M_i\}\!=\{p\in\mathbb{Z}^+\,|\,p\ge M_i \text{ and } p=d_iq+t \text{ for some } t\in T_i \text{ and integer }q\}.$$
Let $M'=\max(\{M_i\,|\,i\in I\}\cup\{p\in P_i^*\,|\,P_i^* \text{ is finite}\})$.
Then, by (4.6.1), it follows that
\begin{equation}\tag{4.6.2}
\begin{aligned}
&\text{for every integer } k\ge M', \text{ the word } w_k \text{ is } M\!\text{-ambiguous if and only if}\\ 
&\text{there exists } i\in I \text{ such that } k=d_iq+t \text{ for some } t\in T_i \text{ and integer } q.
\end{aligned}
\end{equation}

Let $d'=\underset{i\in I}{\prod} d$. By some simple argument, one can see that for any $i\in I$ and integer $k$, we have  $k=d_iq+t$ for some $t\in T_i$ and integer $q$ if and only if $k+d'=d_iq+t$ for some $t\in T_i$ and integer $q$. 
Therefore, by (4.6.2), it holds that for every integer $k\ge M'$, the $M$\!-ambiguity of the words $w_{k+d'}$ and $w_k$ are the same. That is to say, the sequence $\Theta_\varphi$ is eventually periodic.
\end{case1}
In both cases, our conclusion holds.
\end{proof}

Finally, the following generalization holds as a consequence of the above theorem.
\begin{corollary}
Suppose $\Sigma$ is an ordered alphabet. Let $\varphi=\{w_n\}_{n\ge 0}$ be a sequence of words over $\Sigma$ such that for every integer $n\ge 0$, we have $$w_n=x_1a^{k^{(1)}_n}\!x_2a^{k^{(2)}_n}\cdots x_{j}a^{k^{(j)}_n}\!x_{j+1}$$ for some $x_1,x_2,\cdots x_{j+1}\in\Sigma^*$ and $a\in\Sigma$ where
\begin{itemize}[leftmargin=5.5mm]
\item $k^{(i)}_0=1$ for every integer $1\le i\le j$;
\item for every integer $1\le i\le j$, let $e_i$ denote the $j$-tuple with $1$ in the $i^{th}$ coordinate and $0$ elsewhere, and for every integer $n\ge 1$, let $\alpha_n\in\{e_i\,|\,1\le i\le j\}$ and
$$(k^{(1)}_n,k^{(2)}_n,\cdots, k^{(j)}_n)=(k^{(1)}_{n-1},k^{(2)}_{n-1},\cdots, k^{(j)}_{n-1})+\alpha_n;$$
\end{itemize}
If the sequence $\{\alpha_n\}_{n\ge 1}$ is periodic, then the sequence $\Theta_\varphi$ is eventually periodic.
\end{corollary}
\begin{proof}
Suppose the sequence $\{\alpha_n\}_{n\ge 1}$ is periodic, with a period $p$. Then for all integers $1\le n\le p$ and $m\ge 0$, we have $\alpha_{n+mp}=\alpha_n$. Let integers $d_i,(1\le i\le j)$ be such that $(d_1,d_2,\cdots ,d_j)=\sum\limits_{n=1}^{p}\alpha_n$. Next, we need the following observation.
(The validity of the following claim can be easily verified by the reader, thus we omit its technical proof.)
\begin{claim}\label{Claim1}
For every integer $1\le n\le p$, let $\alpha^*_n$ be the $j$-tuple such that $\alpha^*_n=\sum\limits_{i=1}^{n}\alpha_i$ (the addition of tuples is defined element-wise). For all integers $1\le n\le p$ and $1\le i\le j$, let $\mu_{n,i}$ be the value in the $i^{th}$ coordinate of $\alpha^*_n$. Then, for all integers $1\le n\le p$, $1\le i\le j$ and $m\ge 0$, we have $k^{(i)}_{n+mp}=d_im+\mu_{n,i}+1$.
\end{claim}
For all integers $1\le n\le p$ and $m\ge 0$, we have
\begin{equation}\tag{4.9.1}
\begin{aligned}
w_{n+mp}&=x_1a^{k^{(1)}_{n+mp}}x_2a^{k^{(2)}_{n+mp}}\cdots x_{j}a^{k^{(j)}_{n+mp}}x_{j+1}\\
        &=x_1a^{d_1m+\mu_{n,1}+1}x_2a^{d_2m+\mu_{n,2}+1}\cdots x_{j}a^{d_jm+\mu_{n,j}+1}x_{j+1}\\
        &=x_1\underbrace{a^m\cdots a^m}_\text{$d_1$ times}a^{\mu_{n,1}}ax_2\underbrace{a^m\cdots a^m}_\text{$d_2$ times}a^{\mu_{n,2}}a\cdots x_{j}\underbrace{a^m\cdots a^m}_\text{$d_j$ times}a^{\mu_{n,j}}ax_{j+1}
\end{aligned}
\end{equation}
where the second equality holds by Claim~\ref{Claim1}. 

For all integers $0\le n<p$, define the sequence of words $\varphi_n=\{w_{n+mp}\}_{m\ge 0}$. Then, for every integer $0\le n<p$, it follows by $(4.9.1)$ and Theorem~\ref{TheoEventPeriod1} that the corresponding $M$\!-ambiguity sequence $\Theta_{\varphi_n}=\{\theta_{n,t}\}_{t\ge 0}$ is eventually periodic. That is to say, for every integer $0\le n<p$, there exists positive integers $T_n$ and $P_n$ such that for all integers $t\ge T_n$ and $m\ge 0$, we have $\theta_{n,t+mP_n}=\theta_{n,t}$.
Let $T=\max\{T_n\,|\,0\le n<p\}$, then clearly 
\begin{equation}\tag{4.9.2}
\text{for all integers } 0\le n<p,\, t\ge T \text{ and } m\ge 0,\text{ we have } \theta_{n,t+mP_n}=\theta_{n,t}.
\end{equation}

Let $P=p\,\cdot\prod\limits_{n=1}^{p}P_n$. To see that the sequence $\Theta_\varphi=\{\vartheta_t\}_{t\ge 0}$ is eventually periodic, we show that for every integer $t\ge T$, we have $\vartheta_{t+P}=\vartheta_t$. Fix an arbitrary integer $t\ge T$. Let integers $q$ and $0\le r<p$ be such that $t=pq+r$. Then, it can be verified that $\vartheta_t=\theta_{r,q}$, and therefore $\vartheta_{t+P}=\theta_{r,q+\frac{P}{p}}$. It remains to see that since $\frac{P}{p}=\prod\limits_{n=1}^{p}P_n$, it follows by $(4.9.2)$ that $\vartheta_{t+P}=\theta_{r,q+\frac{P}{p}}=\theta_{r,q}=\vartheta_{t}$. Thus our conclusion holds.
\end{proof}

\section{Conclusion}
Unlike the case of binary and ternary alphabets, for larger alphabets, duplication of letters in a word can continuously alter the $M$\!-ambiguity of the resulting words. In fact, by using the main observations and results in Section 3, we have seen that nearly any pattern of $M$\!-ambiguity sequence is attainable. 

As implied in Remark~\ref{RemDrawback}, we are yet to find a print word such that selective repeated duplications of letters in that word could give rise to arbitrary \mbox{$M$\!-ambiguity} sequences starting with the term $U$.
We believe that by further investigation, this would be achievable as well. However, we leave it as an open problem.

The final result in Section 4 shows that repeated duplications of letters of the same type in a word, when done in a periodic manner, give rise to a periodic \mbox{$M$\!-ambiguity} sequence. It remains to see if periodic duplications of different types of letters in a word would lead to the same conclusion. The main complexity would be that the associated systems consist of nonlinear equations and inequalities.

\section*{Acknowledgement}
The authors gratefully acknowledge support for this research by a Research University Grant No.~1001/PMATHS/8011019 of Universiti Sains Malaysia. This study is an extension of the work in \cite{GT17e}.

\bibliographystyle{abbrv}

\end{document}